\documentclass[proceedings]{aofa}

\usepackage[utf8]{inputenc}
\usepackage{subfigure}

\usepackage{amssymb}
\usepackage{amsmath}
\usepackage{enumerate}
\usepackage{tikz}
\usepackage[vlined,titlenotnumbered]{algorithm2e}
\usepackage{stmaryrd}
\usepackage{nicefrac}
\DeclareGraphicsExtensions{.png,.pdf,.jpg}

\usepackage{wrapfig}

\newcommand{\T}{\mathcal{T}}

\newcommand{\Z}{\mathcal{Z}}

\newcommand{\IN}{\mathbb{N}}

\newcommand{\BigO}{\mathcal{O}}

\DeclareMathOperator*{\MSet}{\mbox{\sc MSet}}
\DeclareMathOperator*{\PSet}{\mbox{\sc PSet}}

\newtheorem{theorem}{Theorem}

\newtheorem{proposition}[theorem]{Proposition}

\newtheorem{fact}[theorem]{Fact}

\author{Antoine Genitrini\addressmark{1}}

\title{Full asymptotic expansion for P\'olya structures\thanks{
		This research was partially supported by the ANR MetACOnc project 
		ANR-15-CE40-0014.}}

\date{\today} 

\address{\addressmark{1} Sorbonne Universit\'es, UPMC Univ Paris 06,
  CNRS, LIP6 UMR 7606, 4 place Jussieu 75005 Paris. \\
  Email: {Antoine.Genitrini@lip6.fr}.
}


\begin{document}

\maketitle

\begin{abstract}
In order to obtain the full asymptotic expansion for P\'olya trees,
i.e. rooted unlabelled and non-plane trees, Flajolet and Sedgewick observed
that their specification could be seen as a slight disturbance of the functional equation
satisfied by the Cayley tree function. Such an approach highlights the complicated formal expressions
with some combinatorial explanation. They initiated this process in their book
but they spared the technical part by only exhibiting the first-order approximation.
In this paper we exhibit the university of the method and obtain the full asymptotic expansions
for several varieties of trees.
We then focus on three different varieties of rooted, unlabelled and non-plane trees,
P\'olya trees, rooted identity trees and hierarchies,
in order to calculate explicitly their full singular expansions and asymptotic expansions.
\end{abstract}

\section{Introduction}\label{sec:intro}

By using either Darboux's method or singularity analysis, we easily get the dominant coefficients
of the asymptotic expansions for the number of some specific P\'olya structures; a P\'olya structure being
\begin{wrapfigure}[15]{r}{6cm}
\includegraphics[width=6cm]{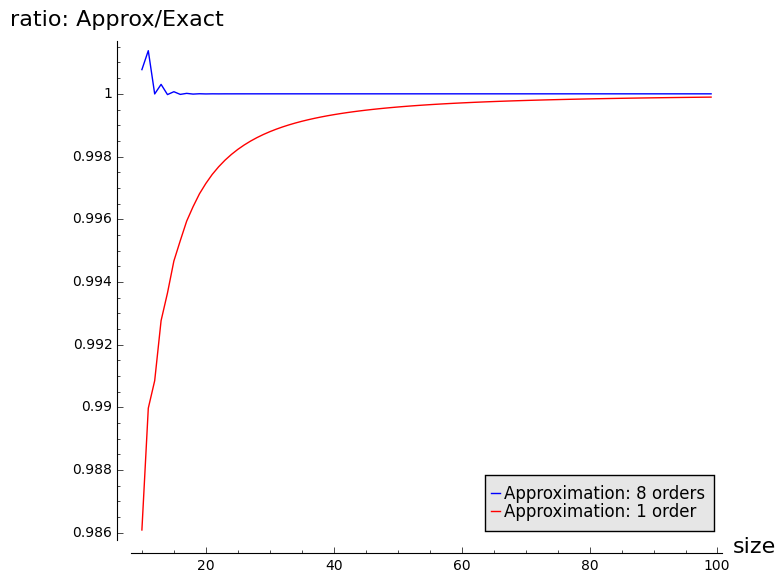}
\caption{Ratio between the approximations and the exact numbers of hierarchies}
\label{fig:approx}
\end{wrapfigure}
decomposable by using some P\'olya operators like the multiset $\MSet$ or
the powerset $\PSet$ constructions. 
For the numbers of hierarchies (a specific class of trees) of size $100$
the relative error between the exact number and the first-order approximation is only around $0.01\%$
(note that it is only $10^{-10}\%$ with an 8-order approximation).
However for small hierarchies, the first-order approximation is not precise: the relative error for the trees
of size $20$ is around $0.3\%$ whereas it is only around $0.0004\%$ with the $8$-order approximation
(cf. Fig.~\ref{fig:approx}).

In a technical report~\cite{Finch03}, Finch
provided recurrence formulas to compute all the coefficients in the asymptotic expansion for P\'olya trees.
He developed there the classical Darboux's method to derive the recurrences
and computed explicitly the five most important coefficients. 

According to Finch's report, Flajolet proposed at that time to study
the fundamental equation given
by the Weierstrass Preparation Theorem as, somehow, a slight disturbance of the functional equation
satisfied by the Cayley tree function. Using this point of view, the procedure to exhibit the full asymptotic
expansion is much more highlighted and the complicated formal expressions can be combinatorially 
understood. Flajolet and Sedgewick initiated this process in their book~\cite[p. 477]{FS09}
in the context of P\'olya trees but they spared the technical part of the proof by only exhibiting the first-order approximation.

In this paper, we explain why such an approach is generic to obtain easily the full asymptotic expansions 
for several varieties of trees. We focus on varieties that
can be seen as a disturbance of the Cayley function in the way
that they can be described by their generating function $T(z)$ as:
\[T(z) = \zeta(z) \exp(T(z)),\]
for some constrained function $\zeta(z)$. For such classes of trees,
we exhibit the full Puiseux (i.e., singular) expansion of the generating series. 
We then compute the generic full asymptotic expansion of the number of trees.
In Section~\ref{sec:calculsExplicites}, we then focus on three different varieties of rooted, unlabelled and non-plane trees.
The first class of trees is the classical set of P\'olya trees
that already appears in the papers of Cayley~\cite{BLW76}, P\'olya~\cite{Polya37}
and Otter~\cite{Otter48}. The generating function of P\'olya trees is easily described with a functional equation
using the multiset construction. By replacing the construction by the powerset operator
we get the class of rooted identity trees, the second class we are interested in. Such trees
are studied, for example, in the work of Harary \emph{et al.} in~\cite{HRS75}.
Finally we deal with hierarchies, i.e., rooted unlabelled non-plane trees without nodes
of arity~$1$. This class has been introduced by Cayley too,
but it is also directly linked to series-parallel networks
in the papers of Riordan and Shannon~\cite{RS42} and Moon~\cite{Moon87}.
In the Section~\ref{sec:approx}, we give numerical approximations for the first coefficients
of the singular and the asymptotic expansions of each specific variety of trees.
We conclude the paper (Section~\ref{sec:concl}) by mentioning several other structures
where our generic approach could be applied directly.


\section{Main results}\label{sec:theo}

For each of the varieties under consideration, 
the fundamental idea consists, from an analytic point of view,
at studying its generating function as a disturbance of the classical Cayley tree function (cf. e.g.~\cite[p. 127]{FS09}).
Let $C(z)$ be the Cayley tree function; it satisfies the functional equation 
\begin{equation}\label{eq:cayley}
  C(z) = z \cdot \exp(C(z)).
\end{equation}
Its dominant singularity is $1/e$ and $C(1/e) = 1$. 
Recall that the Cayley tree function is closely related to the Lambert W function.
Many fundamental results about this classical function
are given in the paper of Corless \emph{et al.}~\cite{CGHJK96}.

In order to obtain generically the full asymptotic expansion of the number of the structures of a variety of trees,
let us first compute the full Puiseux expansion (i.e., the full singular expansion) of the Cayley tree function and then 
study how the disturbance induced by a given variety modifies this behaviour.
Let us recall the definition of Bell polynomials, extensively studied
in Comtet's book~\cite{Comtet74} and denoted by $B_{n,k}(\cdot)$:
\[B_{n,k} \left(x_1, \dots, x_{n-k+1} \right) = \sum_{\scriptsize \begin{array}{c}
 c_1,\dots, c_{n-k+1} \geq 0 \\ \sum_i c_i = k\\ \sum_i i c_i = n
\end{array}}
\frac{n!}{c_1! \cdots c_{n-k+1}!} \left( \frac{x_1}{1!} \right)^{c_1} \cdots \left(\frac{x_{n-k+1}}{(n-k+1)!}\right)^{c_{n-k+1}}.\]
The Bell polynomials appear naturally in Fa\`a di Bruno's formula~\cite{Comtet74}
that states the value of iterated derivatives of the composition of two functions.

\begin{proposition}\label{prop:cayley}
The full Puiseux expansion of the Cayley tree function is
{\footnotesize 
\[
 C(z) \underset{z\rightarrow 1/e}=   1 - \sqrt{2}\sqrt{1-ez} - 
  \sum_{n\geq 2} \left( \sum_{k=1}^{n-1} (-1)^k B_{n-1,k}\left(\frac{1}{3}, \frac{1}{4}, \dots, \frac{1}{n-k+2}\right) \prod_{i=0}^{k-1} (n+2i) \right)
     \frac{2^{n/2}}{n!} \left(1-ez\right)^{n/2},
\]
}
where the functions $B_{n,k}(\cdot)$ are the Bell polynomials.
\end{proposition}
The calculation of the first terms of the singular expansion gives
{\footnotesize
\begin{align*}
C(z) \underset{z\rightarrow 1/e}= & 1 - \sqrt{2}\sqrt{1-ez} +\frac23(1-ez) - \frac{11}{36}\sqrt{2}(1-ez)^{3/2}  +\frac{43}{135}(1-ez)^2 - \\
&  \frac{769}{4320}\sqrt{2}(1-ez)^{5/2} + \frac{1768}{8505}(1-ez)^3   - \frac{680863}{5443200}\sqrt{2}(1-ez)^{7/2} + \BigO\left((1-ez)^{4}\right).
\end{align*}
}
Let us recall that the expansion until $\BigO((1-ez)^{3/2})$ has been derived in~\cite{FS09}.
We prove the full expansion with their approach but with further precision. 
Note that, in the formula of Proposition~\ref{prop:cayley},
the inner sum of $k$ can be factored in the same way as the classical Ruffini-Horner method for polynomial evaluation.
Doing so makes its computations much more efficient.\\

The second step consists in studying the ordinary generating function $T(z)=\sum_{n\geq 0} T_n z^n$ of the tree variety under consideration
as a disturbance of the Cayley tree function.
We follow the approach presented in~\cite[p. 477]{FS09} for P\'olya trees.
We assume the existence of a function $\zeta(z)$ such that
\begin{equation}\label{eq:zeta}
T(z) =  \zeta(z) \cdot \exp(T(z)).
\end{equation}

\begin{theorem}\label{theo:dvt_puiseux}
Let $\T$ be a variety of trees whose generating function is $T(z)$, and $\rho$ be 
its dominant singularity. If the generating function $T(z)$ satisfies the Equation \eqref{eq:zeta},
if the dominant singularity of~$\zeta(z)$ is strictly larger than $\rho$ and if $\zeta^{(1)}(\rho)\neq 0$,
then $T(z)$ satisfies the following full Puiseux expansion
\[T(z) \underset{z\rightarrow \rho}= 1 + \sum_{n\geq 1} t_n \left(1-\frac{z}{\rho}\right)^{n/2},\]
with $t_1 = -\sqrt{2e\rho\zeta^{(1)}(\rho)}$; and, for all $n>1$
\begin{align*}
t_n = - \frac{B(n)}{n!} \left(2e\rho\zeta^{(1)}(\rho)\right)^{n/2} &\; -  
  \sum_{\begin{array}{c} \text{ \scriptsize $\ell = 1$}\\[-1ex] \text{ \scriptsize $n \equiv \ell \bmod 2$} \end{array}}^{\text{ \scriptsize $n-1$}}
(-1)^{(n-\ell)/2}\rho^{n/2} \cdot \frac{B(\ell)}{\ell!} \left(2e\zeta^{(1)}(\rho)\right)^{\ell/2}\\
& \hspace*{10mm}
 \cdot\sum_{r = 1}^{\frac{n-\ell}{2}}  \binom{\ell/2}{r} \frac{1}{(\zeta^{(1)}(\rho))^r} 
\sum_{\begin{array}{c} \text{ \scriptsize $i_1, \dots, i_r \geq 1$}\\ \text{ \scriptsize $\sum_j i_j = \frac{n-\ell}{2}$} \end{array}}
\frac{\zeta^{(i_1+1)}(\rho)}{(i_1+1)!} \cdots \frac{\zeta^{(i_r+1)}(\rho)}{(i_r+1)!},
\end{align*}
where $\zeta^{(i)}(z)$ stands for the $i$th derivative of $\zeta(z)$, $B(1) = 1$, and for all $\ell>1$,
\[B(\ell) =\sum_{k=1}^{\ell-1} (-1)^k B_{\ell-1,k}\left(\frac{1}{3}, \frac{1}{4}, \dots, \frac{1}{\ell-k+2}\right) \prod_{i=0}^{k-1} (\ell +2i).\]
\end{theorem}
\begin{proof}[key idea]
The complete proof 
follows the strategy of Flajolet and Sedgewick.
The main idea is to compose the Puiseux expansion of
$C(z)$ at the singularity $1/e$ and the analytic expansion of $\zeta(z)$ at
the dominant singularity of $T(z)$.
\end{proof}

In Theorem~\ref{theo:dvt_puiseux}, the assumption $\zeta^{(1)}(\rho)\neq 0$ could be
replaced by a weaker assumption that there exists an integer $r>0$
such that $\zeta^{(r)}(\rho)\neq 0$.
Making this weaker assumption would however make the proof a bit more technical
without adding substantial information.

The the first terms of the singular expansion of $T(z)$ are given by
{\scriptsize
\begin{align*}
T(z) \underset{z\rightarrow \rho}= &
1 - \sqrt{2 e \rho \zeta^{(1)}(\rho)} \sqrt{1 - \frac{z}{\rho}} + \frac{2 e \rho \zeta^{(1)}(\rho)}{3} \left(1 - \frac{z}{\rho}\right) 
-  \left(\frac{11\sqrt{2}(e \rho \zeta^{(1)}(\rho))^{3/2}}{36} - \frac{\sqrt{2 e }\rho^{3/2} \zeta^{(2)}(\rho)}{4\sqrt{\zeta^{(1)}(\rho)}}\right) \left(1 - \frac{z}{\rho}\right)^{3/2}\\
& + \left( \frac{43 (e \rho \zeta^{(1)}(\rho))^{2}}{135} - \frac{e (\rho \zeta^{(2)}(\rho))^{2}}{3}\right)  \left(1 - \frac{z}{\rho}\right)^{2} - 
 \left(\frac{769\sqrt{2} (e \rho \zeta^{(1)}(\rho))^{5/2}}{4320} - \frac{11 \sqrt{2} \rho^{5/2}  (e \zeta^{(1)}(\rho))^{3/2}   \zeta^{(2)}(\rho)}{48  \zeta^{(1)}(\rho)}   \right.\\
& \left. - \frac{\sqrt{2} \rho^{5/2}\sqrt{e \zeta^{(1)}(\rho)}}{96} \left(\frac{3 (\zeta^{(2)}(\rho))^{2}}{ (\zeta^{(1)}(\rho))^{2}}
 - \frac{8  \zeta^{(3)}(\rho)}{ \zeta^{(1)}(\rho)}\right)\right)  \left(1 - \frac{z}{\rho}\right)^{5/2}+ \BigO\left(\left(1 - \frac{z}{\rho}\right)^{3}\right).
\end{align*}
}
We are now ready to compute
the full asymptotic expansion for the class $\T$.
\begin{theorem}\label{theo:dvt_nb_trees}
Let $\T$ be a variety of trees whose generating function is $T(z)$, and $\rho$ be 
its dominant singularity. If the generating function $T(z)$ satisfies the Equation \eqref{eq:zeta}, if
the dominant singularity of~$\zeta(z)$ is strictly larger than $\rho$ and if $\zeta^{(1)}(\rho)\neq 0$,
then asymptotically when $n$ tends to infinity,
\[
T_n  \underset{n\rightarrow \infty}\sim \frac{\rho^{-n}}{\sqrt{\pi n^3}} \sum_{\ell \geq 0} \frac{1}{n^\ell}\cdot \left( \sum_{r = 1}^{\ell+1} Q_r R_{\ell+1-r}\right),
\]
where
\[Q_r =   \sum_{j=0}^{r-1} (-1)^{j+1} t_{2j+1} 
         \sum_{\begin{array}{c} \text{\scriptsize $\ell_0, \dots, \ell_j \geq 1$}\\ \text{\scriptsize $\sum_i \ell_i= r$} \end{array}}
           \prod_{i=0}^j \left( i + \frac12 \right)^{\ell_j} \qquad \text{for all } r>0; \]
with the sequence $(t_i)$ defined in Theorem~\ref{theo:dvt_puiseux}, 
$R_0 = 1$ and
\[
 R_{\ell} = \sum_{\begin{array}{c} \text{\scriptsize $r = 1$}\\ \text{\scriptsize $r \equiv \ell \bmod 2 $} \end{array}}^\ell
     \sum_{\begin{array}{c} \text{\scriptsize $k_1, \dots, k_r \geq 1$}\\ \text{\scriptsize $\sum_j k_j  = \frac{\ell+r}{2}$} \end{array}}
       \prod_{i=1}^r \frac{(2^{-2k_i}-1)  \sum_{s=0}^{2k_i} \frac{1}{s+1} \sum_{j=0}^s (-1)^j \binom{s}{j} j^{2k_i} }{( \ell -2 k_1 - \cdots - 2k_{i-1} + i-1) k_i} \qquad \text{for all } \ell > 0.
\]
\end{theorem}

In particular, the first few terms in the asymptotic expansion of $T_n$ are given by
{\footnotesize
\begin{align*}
T_n  \underset{n\rightarrow \infty}= \frac{\rho^{-n}}{\sqrt{\pi n^3}} &\left(
-\frac{t_1}{2} -\frac{3(t_1 - 4t_3)}{16n} - \frac{5(5t_1 - 72t_3 + 96t_5)}{256n^2}
- \frac{105(t_1 - 44t_3 + 160t_5 - 128t_7)}{2048n^3}  \right.\\
& \hspace*{3mm} \left. - \frac{21(79t_1 - 10800t_3 + 81600t_5 - 161280t_7 + 92160t_9)}{65536n^4} + \BigO\left(\frac{1}{n^5}\right)\right),
\end{align*}
}
where the $t_i$'s are given in the Theorem~\ref{theo:dvt_puiseux}.

%
%

\section{Different varieties of rooted unlabelled and non-plane trees}\label{sec:calculsExplicites}

In the following three sections, we will show how both Theorems~\ref{theo:dvt_puiseux} and~\ref{theo:dvt_nb_trees}
directly apply to three families of trees, namely the P\'olya trees, the rooted identity trees and the hierarchies.
In each of these sections, we will use the same notations $\T$, $T(z)$ and $\zeta(z)$
to refer to the family of considered trees.

For each of the three examples, we proceed in two steps.
First we focus on efficient recurrences in order to
compute the first numbers of the sequence $(T_n)_{n\in \IN}$
that encodes for each positive integer $n$ the number of trees of size $n$.
Second, by using the numerical procedure given in~\cite[p. 477]{FS09},
we compute an approximation of the dominant singularity of $T(z)$.

Finally, at the end of the section, we exhibit two Tables~\ref{table:approxPuiseux} and~\ref{table:approxNBTrees}
to compare the numerical approximations (according to each class of trees) of the coefficients
given in the Theorems~\ref{theo:dvt_puiseux} and~\ref{theo:dvt_nb_trees}.
We also exhibit the typical gain in the relative error obtained by using a more precise asymptotic approximation.

\subsection{P\'olya trees}\label{sec:polya}

A P\'olya tree is a rooted unlabelled and non-plane tree. Let us denote by $\T$ the set of P\'olya trees.
It satisfies the following unambiguous specification :
\[\T = \Z \times \MSet \T,\]
because a P\'olya tree is by definition a root, specified by $\Z$ (of size $1$),
followed by a multiset of P\'olya trees (we refer the reader to~\cite{FS09} for more details).
By the \emph{symbolic method} (cf.~\cite{FS09}), we get 
\begin{equation}\label{eq:Polya_fun_eq}
  T(z) = z \exp\left( \sum_{i>0} \frac{T(z^i)}{i}\right),
\end{equation}
with $T(z)$ being the ordinary generating function enumerating $\T$.
The latter formula already appears
in P\'olya's paper~\cite{Polya37} and has been sketched
by Cayley (\cite[p. 67]{BLW76}) as an introduction to the 
counting theory for unlabelled objects. This method takes into account symmetries of the objects
and thus quantifies isomorphisms.
We have a classical alternative definition: cf. e.g.~ \cite[p. 71]{FS09}.
\begin{equation}\label{eq:Polya_produit}
  T(z) = z \cdot \prod_{n>0} \frac{1}{(1-z^n)^{T_n}},
\end{equation}
with $T_n$ the number of trees of size $n$ in $\T$.
Some combinatorial arguments, given in~\cite[p. 27--30]{FS09}, prove that both definitions are equivalent.
From the latter Equation~\eqref{eq:Polya_produit}, we deduce a recurrence for the sequence $(T_n)_{n\in \IN}$
for P\'olya trees.
\begin{fact}\label{rec:otter}
The sequence $(T_n)_{n\in\IN}$ enumerating P\'olya trees satisfies
\[T_n = \left\lbrace 
    \begin{array}{l l}
        n & \text{if } n \in\{0,1\} \\ 
        \displaystyle{\frac{1}{n-1} \sum_{i=1}^{n-1}  i T_i \left( \displaystyle{\sum_{m=1}^{\lfloor \frac{n-1}{i}\rfloor} T_{n-mi} } \right)} & \text{if } n>1.
    \end{array}
    \right. \]
\end{fact}
This result is given as an exercise by Knuth~\cite[p. 395]{Knuth97vol1}.
Furthermore, Otter~\cite{Otter48} proved a very similar recurrence for unrooted trees.
The first values of the sequence, given in OEIS\footnote{OEIS: On-line Encyclopedia of Integer Sequences} sequence A000081,
are
\[0, 1, 1, 2, 4, 9, 20, 48, 115, 286, 719, 1842, 4766, 12486, 32973, 87811, \dots\]
The number of P\'olya trees from each size from $1$ to $n$ can be
computed in $\BigO(n^2)$ arithmetic operations (by using memoization).

\begin{proof}[of Fact~\ref{rec:otter}]
Several authors, in particular,
Flajolet and Sedgwick obtained such a recurrence by using the logarithmic derivative of $T(z)$:
for all $n>1$
\[z\frac{T'(z)}{T(z)} = 1 + \sum_{n>0} T_n \frac{n z^{n}}{1-z^n}.\]
We rewrite this equation as
\[z T'(z) = \left( 1 + \sum_{n>0} n T_n \frac{z^{n}}{1-z^n} \right) T(z).\]
Extracting the $n$-th coefficient of the generating functions gives:
\[ n T_n = T_n + \sum_{i=1}^{n-1}\left( [z^i] \sum_{m>0} m T_m \frac{z^m}{1-z^m} \right) T_{n-i}.\]
Since $[z^k] (1-z^m)^{-1}$ equals $1$ if $m$ divides $k$ and 0 otherwise, 
we get
\[ (n-1) T_{n} = \sum_{i=1}^{n-1} \left( \sum_{m | i} m T_m \right) T_{n-i}.\]
The notation $m | i$ corresponds to the condition that the integer $m$ divides the integer $i$.
The stated formula is obtained by interchanging the two sums.
\end{proof}

By using Flajolet and Sedgewick's numerical procedure (cf.~\cite[p. 477]{FS09}) with $n=200$ terms,
we get the following $50$-digits approximation of $\rho$:
\[\rho \approx 0.33832185689920769519611262571701705318377460753297\ldots\]

We are now interested in the full Puiseux expansion of the generating function of P\'olya trees.
In view of Equations~\eqref{eq:zeta} and~\eqref{eq:Polya_fun_eq}, we define have
\begin{equation}\label{eq:polya}
T(z) =\zeta(z) \cdot \exp(T(z)), \qquad \text{where } \zeta(z) =  z\cdot \exp\left(\sum_{n\geq 2} \frac{T(z^n)}{n}\right).
\end{equation}

\begin{fact}\label{prop:polya}
The function $\zeta(z)$ defined for P\'olya trees satisfies the assumptions of the 
Theorems~\ref{theo:dvt_puiseux} and~\ref{theo:dvt_nb_trees}.
\end{fact}
This fact has already been proved by Cayley as mentioned in~\cite[p. 67]{BLW76}.
We recall here the arguments given in~\cite[p. 477]{FS09}.

\begin{proof}
The definition of $\zeta(z)$ given in Equation~\eqref{eq:polya}
implies that its dominant singularity is $\sqrt{\rho}$,
(with the constant $\rho$ being the dominant singularity of $T(z)$).
Since $1/e$ is the dominant singularity of the Cayley tree function $C(z)$
and $[z^n] T(z) > [z^n] C(z)$ (by using Equation~\eqref{eq:Polya_fun_eq})
for $n$ sufficiently large, we get $\rho \leq 1/e$.
Thus $\sqrt{\rho} > \rho$ and we finally infer that the function $\zeta(z)$
is analytic beyond the disc of convergence of $T(z)$.
Finally we easily get $\zeta'(\rho) >0$.
\end{proof}

Theorem~\ref{theo:dvt_puiseux} and the above approximation for $\rho$ give
the first coefficients for the Puiseux expansion of P\'olya trees presented in the Table~\ref{table:approxPuiseux}.
The computations of the numbers $t_i$'s have been done with an approximation of the function $\zeta(z)$,
computed with the truncation of the series $T(z)$ after the $100$-th first coefficients.
Experimentally, it seems that the accuracy is actually much larger than the $20$ digits given in Table~\ref{table:approxPuiseux}.

Finally the previous approximations and the result of Theorem~\ref{theo:dvt_nb_trees} give
{\footnotesize
\begin{align*}
T_n  \underset{n\rightarrow \infty}= \frac{\rho^{-n}}{\sqrt{\pi n^3}} &\left(
0.7797450101873204419\ldots +\frac{0.07828911261061096133\ldots}{n} + \frac{0.3929402676631860168\ldots}{n^2} + \right. \\
& \hspace*{3mm} \left. \frac{1.537879315978838092\ldots}{n^3} +  \frac{8.200844090435596194\ldots}{n^4} + \BigO\left(\frac{1}{n^5}\right)\right).
\end{align*}
}
Note that, from here, it is then easy to get back the first evaluations exhibited by Finch~\cite{Finch03}.

\subsection{Rooted identity trees}

A rooted identity tree is a rooted unlabelled (non-plane) tree for which the only automorphism preserving the root node is the identity.
Harary~\emph{et al.} studied this class of trees in~\cite{HRS75}. In his book~\cite{Finch03book}, Finch also mentions this class.
Intuitively, whereas a P\'olya tree can be seen as a root followed by a multiset of P\'olya trees,
a rooted identity tree can be seen as a root followed by a set of rooted identity trees (i.e., no repetition is allowed).
Let us denote by $\T$ the set of rooted identity trees.
It satisfies the following unambiguous specification 
\[\T = \Z \times \PSet \T.\]
The symbolic method gives the functional equation
\[T(z) = z \exp\left( \sum_{i>0} (-1)^{i-1}\frac{T(z^i)}{i}\right).\]
An equivalent formula for the function $T(z)$ is
\[ T(z) = z \cdot \prod_{n>0} (1+z^n)^{T_n}.\]
In order to obtain an efficient recurrence relation satisfied by the numbers of rooted identity tree, 
we use the same strategy as above (for P\'olya trees), and thus obtain:
\begin{proposition}\label{rec:identity}
The sequence $(T_n)_{n\in\IN}$ enumerating rooted identity trees satisfies
\[T_n = \left\lbrace 
    \begin{array}{l l}
        n & \text{if } n \in\{0,1\} \\ 
       \displaystyle{ \frac{1}{n-1} \sum_{i=1}^{n-1}  i T_i \left( \displaystyle{\sum_{m=1}^{\lfloor \frac{n-1}{i}\rfloor} (-1)^{m+1} T_{n-mi} } \right)} & \text{if } n>1.
    \end{array}
    \right. \]
\end{proposition}
The first values of the sequence, see in OEIS A004111, are
\[0, 1, 1, 1, 2, 3, 6, 12, 25, 52, 113, 247, 548, 1226, 2770, 6299, \dots\]
The number of rooted identity trees from each size from $1$ to $n$ can be computed in $\BigO(n^2)$ arithmetic operations.
Once we are able to compute efficiently the first numbers $T_n$
we can estimate the dominant singularity of $T(z)$ to be approximately
\[\rho \approx 0.39721309688424004148565407022739873422987370995276\ldots \]
Obviously this dominant singularity is larger than the one for P\'olya trees because there are less 
rooted identity trees than P\'olya trees.

To describe $T(z)$ like in Equation~\eqref{eq:zeta}, we get
$\zeta(z) =  z\cdot \exp\left(\sum_{n\geq 2} (-1)^{n-1} \frac{T(z^n)}{n}\right)$.

\begin{proposition}\label{prop:nb_identity}
The function $\zeta(z)$ defined in the context of rooted identity trees satisfies the assumptions of the 
Theorems~\ref{theo:dvt_puiseux} and~\ref{theo:dvt_nb_trees}.
\end{proposition}
The approximations of the first coefficients of the Puiseux expansion for rooted identity
trees are given in the Table~\ref{table:approxPuiseux}.
The second Table~\ref{table:approxNBTrees}
gives the approximations of the asymptotic expansion of $T_n$:

{\footnotesize
\begin{align*}
T_n  \underset{n\rightarrow \infty}= \frac{\rho^{-n}}{\sqrt{\pi n^3}} &\left(
0.6425790797442694714\ldots -\frac{0.1851197977766337056\ldots}{n} - \frac{0.4272427290060978745\ldots}{n^2}  \right. \\
& \hspace*{3mm} \left. - \frac{2.255455568987212079\ldots}{n^3} -  \frac{16.60970953335647846\ldots}{n^4} + \BigO\left(\frac{1}{n^5}\right)\right).
\end{align*}
}
It seems that these numbers do not appear elsewhere in the literature.

\subsection{Hierarchies}

A hierarchy is a rooted unlabelled and non-plane tree with no node of arity~1.
The size notion for hierarchies is the number of leaves.  This class already appears in the work of Cayley
(cf. \cite[p. 43]{BLW76}. 
Using the notations from~\cite[p. 72]{FS09} for hierarchies, we have both following specification
and functional equation for its generating function
\[\T = \Z + {\mbox{\sc MSet}_{\geq 2}} \T, \qquad
T(z) = \frac{1}{2} \left( z -1 + \exp \left( \sum_{i>0} \frac{T(z^i)}{i} \right)  \right).\]

Again, we obtain a recurrence formula that computes the numbers $T_n$.
\begin{proposition}\label{rec:hierarchies}
The sequence $(T_n)_{n\in\IN}$ enumerating hierarchies satisfies
\[T_n = \left\lbrace 
    \begin{array}{l l}
        n & \text{if } n \in\{0,1\} \\ 
        \displaystyle{ \frac{1}{n} \sum_{\text{\scriptsize $\begin{array}{c}  m | n \\  m \neq n \end{array}$}}
            m T_m + \frac{2}{n}\left(\sum_{i=1}^{n-1}  i T_i 
            \sum_{m=1}^{\lfloor \frac{n-1}{i}\rfloor} T_{n-mi} -\frac12 \delta_{\{n-mi=1\}}  \right)\  }
       & \text{if } n>1,
    \end{array}
    \right. \]
with the notation $\delta_{\{n-mi=1\}}$ evaluates to $1$ if $n-mi=1$ and to $0$ otherwise.
\end{proposition}
The first values of the sequence, see in OEIS A000669, are given by
\[0, 1, 1, 2, 5, 12, 33, 90, 261, 766, 2312, 7068, 21965, 68954, 218751, 699534, \dots\]
They are stored (there the sequence is shifted by 1).
We note that in this context, we cannot easily simplify the recurrence in order to avoid a sum
over the divisors of $n$ (for $T_n$). However here, the sum is not inside another one, thus 
the complexity (in the number of arithmetic operations)
to compute $T_n$ is quadratic.
We estimate the dominant singularity of $T(z)$ to be approximately
\[\rho \approx 0.28083266698420035539318755911632333333736599643391\ldots \]

In order to fall under the framework described by Equation~\eqref{eq:zeta},
we need to consider the generating function $\tilde{T}(z) = T(z) - \frac12(1-z)$.
The two generating functions $T(z)$ and $\tilde{T}(z)$ have the same dominant singularity.
Thus we get 
\[\tilde{T}(z) = \zeta(z) \cdot \exp(\tilde{T}(z)),\]
with
\[\zeta(z) = \frac{1}{2}  \exp\left( \frac12(1-z) + \sum_{n\geq 2} \frac{T(z^i)}{i}\right).\]

\begin{proposition}\label{prop:nb_hierrchies}
The function $\zeta(z)$ defined in the class of objects associated to $\tilde{T}(z)$
satisfies the assumptions of the Theorems~\ref{theo:dvt_puiseux} and~\ref{theo:dvt_nb_trees}.
\end{proposition}
It remains to slightly modify the 2 first coefficients in the singular expansion of $\tilde{T}(z)$
to obtain the singular expansion of $T(z)$ and fill both Tables~\ref{table:approxPuiseux}.
and~\ref{table:approxNBTrees}. In particular we get
{\footnotesize
\begin{align*}
T_n  \underset{n\rightarrow \infty}= \frac{\rho^{-n}}{\sqrt{\pi n^3}} &\left(
0.3658015862381119375\ldots -\frac{0.2409833212579280352\ldots}{n} - \frac{0.3678657493849431861\ldots}{n^2} \right. \\
& \hspace*{3mm} \left. - \frac{0.9991064877914853523\ldots}{n^3} -  \frac{4.137777553476907813\ldots}{n^4} + \BigO\left(\frac{1}{n^5}\right)\right).
\end{align*}
}
It seems that these numbers do not appear elsewhere in the literature.

Let us conclude this section on hierarchies by mentioning the OEIS sequence A000084, that is directly related.
It counts the number of series-parallel networks with $n$ unlabelled edges;
both generating functions are essentially the same (up to a simple factor).
We thus get the Puiseux expansions and the asymptotic expansion for these objects as a by-product.

\subsection{Approximations}\label{sec:approx}

In order to obtain the following approximations for the coefficients in the Puiseux expansions or for the asymptotic expansions
of the numbers of trees, we have used the open-source mathematics software \emph{Sage}~\cite{Sage}
and the Python library \emph{MPmath}~\cite{Mpmath} for some specific high precision calculations.

The first table synthesises the first elements of the sequences $(t_n)_{n\in\IN}$
satisfying the Puiseux expansions for the previous P\'olya structures:
\[T(z) = \sum_{n\geq 0} t_n \left(1 - \frac{z}{\rho}\right)^{n/2}.\]

\begin{table}[!h]
\begin{center}
{\scriptsize
\renewcommand{\arraystretch}{1.1}
\begin{tabular}{| l | l | l | l |}
\hline 
Coeff. & P\'olya trees & Rooted identity trees & Hierarchies \\
\hline
$t_{0}$ & $1.000000000000000000$ & $1.000000000000000000$ & $0.6404163334921001777$ \\ \hline $t_{1}$ & $-1.559490020374640884$ & $-1.285158159488538943$ & $-0.7316031724762238750$ \\ \hline $t_{2}$ & $0.8106697078826992796$ & $0.5505438316333229659$ & $0.03799806716699161541$ \\ \hline $t_{3}$ & $-0.2854870216128456058$ & $-0.5681159369076463432$ & $0.1384103018915147449$ \\ \hline $t_{4}$ & $0.1653723657120838943$ & $0.4261261857916583247$ & $-0.07387395031732463851$ \\ \hline $t_{5}$ & $-0.3424599704021542007$ & $-0.1312888430707878210$ & $-0.05428300802019698042$ \\ \hline $t_{6}$ & $0.3174072259465285628$ & $0.1224152517144394163$ & $0.03800381072191918081$ \\ \hline $t_{7}$ & $-0.1077788002916310083$ & $-0.3225499663026797778$ & $0.03109684705422999274$ \\ \hline $t_{8}$ & $0.06138495705583510410$ & $0.2539454170234272677$ & $-0.02381831461193008886$ \\ \hline $t_{9}$ & $-0.1952123835975564636$ & $0.04875363678533678081$ & $-0.02078556533052714092$ \\ \hline $t_{10}$ & $0.2059848312779074186$ & $-0.00002800001023286558041$ & $0.01666265537126027377$ \\ \hline $t_{11}$ & $-0.05272470849819056138$ & $-0.3631594631270670335$ & $0.01611178365047090583$ \\ \hline $t_{12}$ & $0.01702656875495366861$ & $0.2637344037695510765$ & $-0.01295368177079785790$ \\ \hline $t_{13}$ & $-0.1523706243663253961$ & $0.2617035123807709629$ & $-0.01338408339711046374$ \\ \hline $t_{14}$ & $0.1737028832998504627$ & $-0.1368754575043169801$ & $0.01075691931570711729$ \\ \hline $t_{15}$ & $-0.01447370373952704466$ & $-0.5927534134371262366$ & $0.01183388780152404393$ \\ \hline $t_{16}$ & $-0.02189951761121556237$ & $0.3911340105112945142$ & $-0.009441457380326882677$ \\ \hline $t_{17}$ & $-0.1445471935709097045$ & $0.6832510269350502136$ & $-0.01084956346194149131$ \\ \hline $t_{18}$ & $0.1760771088850177779$ & $-0.3902593892984113718$ & $0.008607637481105329431$ \\ \hline 
\end{tabular} 
}
\caption{Approximation of the Puiseux expansions for P\'olya trees, rooted identity trees and hierarchies}
\label{table:approxPuiseux}
\end{center}
\end{table}

The following Table~\ref{table:approxNBTrees} contains the first
numbers $(\tau_n)_{n\in\IN}$ satisfying the asymptotic expansions
for the previous P\'olya structures:
\[T_n \underset{n\rightarrow \infty}\sim \frac{\rho^{-n}}{\sqrt{\pi n^3}} \sum_{i\geq 0} \frac{\tau_i}{n^i}.\]
\begin{table}[!h]
\begin{center}
{\scriptsize
\renewcommand{\arraystretch}{1.1}
\begin{tabular}{| l | l | l | l |}
\hline 
Coeff. & P\'olya trees & Rooted identity trees & Hierarchies \\
\hline 
$\tau_{0}$ & $0.7797450101873204419$ & $0.6425790797442694714$ & $0.3658015862381119375$ \\ \hline $\tau_{1}$ & $0.07828911261061096133$ & $-0.1851197977766337056$ & $0.2409833212579280352$ \\ \hline $\tau_{2}$ & $0.3929402676631860168$ & $-0.4272427290060978745$ & $0.3678657493849431861$ \\ \hline $\tau_{3}$ & $1.537879315978838092$ & $-2.255455568987212079$ & $0.9991064877914853523$ \\ \hline $\tau_{4}$ & $8.200844090435596194$ & $-16.60970953335647846$ & $4.137777553476907813$ \\ \hline $\tau_{5}$ & $57.29291473494343825$ & $-157.9003693373302727$ & $23.43410248921570084$ \\ \hline $\tau_{6}$ & $503.0445050262735854$ & $-1840.110517359351172$ & $170.1188811511555370$ \\ \hline $\tau_{7}$ & $5359.600933884326064$ & $-25387.34869954017854$ & $1514.745295656330186$ \\ \hline $\tau_{8}$ & $67342.06920114653067$ & $-404610.0663959841556$ & $16007.82637588106931$ \\ \hline $\tau_{9}$ & $975425.4970695924728$ & $-7.313377058487246593e6$ & $195812.3506172274875$ \\ \hline $\tau_{10}$ & $1.599693249293173348e7$ & $-1.477949138517813328e8$ & $2.719234685827618831e6$ \\ \hline $\tau_{11}$ & $2.928225313353392698e8$ & $-3.301794456762036735e9$ & $4.222444465223140109e7$ \\ \hline $\tau_{12}$ & $5.914523441293936053e9$ & $-8.080229604228356791e10$ & $7.243861962702191648e8$ \\ \hline $\tau_{13}$ & $1.305991927898973201e11$ & $-2.149826267241085239e12$ & $1.359774926415692519e10$ \\ \hline $\tau_{14}$ & $3.128498399789526502e12$ & $-6.179075814699061934e13$ & $2.770908644498957323e11$ \\ \hline $\tau_{15}$ & $8.078305401468914384e13$ & $-1.908151484770832703e15$ & $6.089496262810801422e12$ \\ \hline $\tau_{16}$ & $2.236301680891647428e15$ & $-6.301063280436556255e16$ & $1.435269254893331074e14$ \\ \hline $\tau_{17}$ & $6.605960869699262787e16$ & $-2.215767775919040241e18$ & $3.610881990157578400e15$ \\ \hline $\tau_{18}$ & $2.073828085209932615e18$ & $-8.267080545525264413e19$ & $9.656755540184967275e16$ \\ \hline 
\end{tabular} 
}
\caption{Asymptotic expansion of the number of P\'olya trees, rooted identity trees and hierarchies}
\label{table:approxNBTrees}
\end{center}
\end{table}

It is interesting to note that, in Table~\ref{table:approxNBTrees}, for $n$ sufficiently large
and due to the sign of the values of the $(\tau_i)$,
all truncations after the $n$th term in the full expansions (for $n=1\dots 17$) correspond to lower bounds for the case of 
P\'olya trees and hierarchies and all of them are upper bounds for rooted identity trees.

\begin{table}[!h]
\begin{center}
{\scriptsize
\renewcommand{\arraystretch}{1.4}
\begin{tabular}{| l | l | l | l | l | l | l |}
\hline 
Size & $10$ & $20$ & $50$ & $100$ & $200$ & $500$ \\
\hline 
Order-1 approximation & $1.391\cdot 10^{-2}$ & $2.859\cdot 10^{-3}$ & $4.204\cdot 10^{-4}$ & $1.027\cdot 10^{-4}$ & $2.540\cdot 10^{-5}$ & $4.039\cdot 10^{-6}$ \\ 
 \hline
Order-4 approximation & $1.039\cdot 10^{-3}$ & $3.448\cdot 10^{-5}$ & $2.383\cdot 10^{-7}$ & $6.872\cdot 10^{-9}$ & $2.071\cdot 10^{-10}$ & $2.078\cdot 10^{-12}$ \\ 
\hline 
Order-8 approximation & $7.722\cdot 10^{-4}$ & $3.369\cdot 10^{-6}$ & $3.822\cdot 10^{-10}$ & $6.195\cdot 10^{-13}$ & $1.123\cdot 10^{-15}$ & $2.611\cdot 10^{-18}$ \\ 
\hline
\end{tabular} 
}
\caption{Relative error induced by approximations for hierarchies}
\label{table:approxError}
\end{center}
\end{table}

Finally, by using only 20 digits of precision in our approximations
of the values $\zeta^{(r)}(\rho)$'s we cannot
hope to obtain a better approximation than the one of order 8 (Table~\ref{table:approxError})
for the number of large trees (i.e. with size larger than $500$).

\section{Conclusion}\label{sec:concl}

The strength of the approach presented here is its universality.
We have shown, in full detail, how it applies to P\'olya trees,
rooted identity trees and hierarchies
but many other examples fill in our framework.
\begin{enumerate}
	\item \emph{Rooted oriented trees and series-reduced planted trees.}
	The OEIS sequences A000151 and A001678 can be directly studied.
	
	\item \emph{Series-parallel networks.}
	In the context of~\cite{RS42}, \cite{Moon87} and~\cite{Finch03seriesparallel}
	we get back several generating functions (listed in OEIS
	A058385, A058386 and A058387)
	that can be studied in the same vein as hierarchies. Let us recall that many links between
	trees and series-parallel graphs have already been exhibited, thus
	the fact that the behaviours of their generating series are analogous is not a surprise.
	
	\item \emph{Phylogenetic trees and also total partitions.}
	The OEIS sequence A000311, counting phylogenetic trees and also total partitions
	that are labelled objects, can also be analysed with our technique. 
	Note here that the function $\zeta(z)$ does not explicitly depend on $T(z)$
	and thus every derivative is explicit. Just put a factor $n!$ in front of $T_n$
	to obtain its full asymptotic expansion. We thus exhibit the polynomials
	whose existence has been stated in~\cite[p. 224]{Comtet74}.
	
	\item \emph{The unrooted versions of the previous rooted trees.}
	With some further work, we are able to exhibit the full
	asymptotic expansion of the unrooted versions of the previous rooted trees we were
	interested in. In fact their generating functions $P(z)$ satisfy some equation
	of the form 
	{\footnotesize
	\[P(z) = T(z) -\frac12 T^2(z) + \frac12T(z^2).\]
	}
	Since we have the full Puiseux expansion of the series $T(z)$, we can compute the one
	of the series $P(z)$.
	Some examples of such series correspond to the following sequences A000055, A000238, A000014\dots.
	An open question would be to be able to write a functional equation for $P(z)$ as a
	disturbance of the Cayley tree function, and then to use directly an analogous approach as the one studied 
	in Section~\ref{sec:theo}. There, we would get $\zeta^{(1)} = 0$ since we know that these trees are unrooted.
\end{enumerate}

\vspace*{0.5cm}
{\small
\noindent \textbf{Acknowledgements.} 
	The author is very grateful to Cécile Mailler for the carefully reading of this manuscript
	and to the anonymous referees for the suggested improvements.
}

\newpage
\bibliographystyle{alpha}
\bibliography{polya}

\end{document}